\documentclass[12pt]{amsart}
\usepackage{amsfonts, amssymb, amsthm, amsmath, enumerate, hyperref, mathbbol, mathtools}
\usepackage{multirow, hyperref}
\newtheorem{thm}{Theorem}[section]

\newtheorem{cor}[thm]{Corollary}
\newtheorem{lemma}[thm]{Lemma}
\newtheorem{prop}[thm]{Proposition}

\DeclareMathOperator{\Res}{Res }

\DeclarePairedDelimiter\floor{\lfloor}{\rfloor}
\newcommand{\miss}[1][S]{\Upsilon_{#1}}
\newcommand{\misss}[1][S]{\Upsilon_{#1}^\sharp}
\newcommand{\cms}[1][S]{\kappa_{#1}}
\newcommand{\cmss}[1][S]{\kappa_{#1}^\sharp}
\newcommand{\CC}{\mathbb{C}}

\newcommand{\HH}{\mathcal{H}}
\newcommand{\NN}{\mathbb{N}}
\newcommand{\ZZ}{\mathbb{Z}}

\newcommand{\SL}{{\text {\rm SL}}}

\newcommand{\fp}[3]{f^{(#1)}_{#2,#3}}
\newcommand{\gp}[3]{g^{(#1)}_{#2,#3}}
\newcommand{\la}[1][p]{\Lambda_{#1}}
\newcommand{\abs}[1]{\left|#1\right|}

\begin{document}
\title{Zagier duality for level $p$ weakly holomorphic modular forms}
\author{Paul Jenkins}\email{jenkins@math.byu.edu}
\author{Grant Molnar}\email{gmolnar@mathematics.byu.edu}\thanks{
	Brigham Young University, Provo, UT. This work was partially supported by a grant from the Simons Foundation, ($\# 281876$ to Paul Jenkins).}

\begin{abstract}
	We prove Zagier duality between the Fourier coefficients of canonical bases for spaces of weakly holomorphic modular forms of prime level $p$ with $11 \leq p \leq 37$ with poles only at the cusp at $\infty$, and special cases of duality for an infinite class of prime levels. We derive generating functions for the bases for genus 1 levels.
\end{abstract}

\maketitle

\section{Introduction}

In 2002, Zagier \cite{Zagier} proved that the Fourier coefficients of two sequences of half-integral weight modular forms exhibit a curious duality: the $m^{th}$ coefficient of the $n^{th}$ form in one sequence is the negative of the $n^{th}$ coefficient of the $m^{th}$ form in the other sequence. To prove this, Zagier used a bivariate generating function for the two sequences of forms. Bringmann and Ono extended Zagier's results by proving duality theorems for harmonic Maass forms and Poincar\'{e} series of level 4 and half-integral weight \cite{Bringmann}. Likewise, Rouse \cite{Rouse}, Choi \cite{Choi}, and Zhang \cite{Zhang} showed that duality holds for certain Hilbert modular forms and forms with quadratic character.

In 2007, Duke and the first author discovered Zagier duality in integral weight weakly holomorphic modular forms \cite{Jenkins1}, again using a generating function. Let $q = e^{2 \pi i z}$, and denote by $M_k^!(N)$ the space of weakly holmorphic modular forms of level $N$. Let $\{f_{k, m} = q^{-m} + \sum\limits_n a_k(m,n) q^n \}_m$ be the reduced row echelon basis for $M_k^!(1)$. \begin{thm}[\cite{Jenkins1}, Theorem 2]\label{JenkinsDuke}
	For any even integer $k$ we have the generating function \[\sum\limits_{m \geq - \ell} f_{k,m}(\tau)q^m = \frac{f_k(\tau) f_{2 - k}(z)}{j(z) - j(\tau)}\] where $f_k = \Delta^\ell E_{k'}$ with $k = 12 \ell + k'$ and $k' \in \{0, 4, 6, 8, 10, 14\}$.
\end{thm} \noindent Here $\Delta$ is the discriminant modular form and $E_{k'}$ is the weight $k'$ Eisenstein series. As a corollary, the theorem gives duality between basis elements of weight $k$ and weight $2 - k$. \begin{cor}[\cite{Jenkins1}, Corollary 1]\label{JenkinsDuke2}
Let $k$ be an even integer. For all integers $m, n$ the equality \[a_k(m, n) = - a_{2 - k}(n,m)\] holds for the Fourier coefficients of the modular forms $f_{k,m}$ and $f_{2 - k, n}$.
\end{cor}

For $N$ a natural number, denote by $M_k^\sharp(\Gamma_0(N)) = M_k^\sharp(N)$ the space of weakly holomorphic modular forms of level $N$ that are holomorphic away from $\infty$, and denote by $S_k^\sharp(\Gamma_0(N)) = S_k^\sharp(N)$ the space of weakly holomorphic modular forms of level $N$ which vanish at every cusp other than $\infty$. Let $\{ \fp Nkm (z) = q^{-m} + \sum\limits_{n > -m} a_k^{(N)}(m, n) q^n \}_m$ be the reduced row echelon basis for $M_k^\sharp(N)$ and let $\{ \gp Nkm (z) = q^{-m} + \sum\limits_{n > -m} b_k^{(N)}(m, n) q^n \}_m$ be the reduced row echelon basis for $S_k^\sharp(N)$. We define $a_k^{(N)}(m,n)$ to be $0$ if there is no $q^n$ term in $f_{k,m}^{(N)}$, or if $f_{k,m}^{(N)}$ does not exist, and we define $a_k^{(N)}(m,-m)$ to be $0$ even though the coefficient of $q^{-m}$ is 1. As the only cusp of $\SL_2(\ZZ) = \Gamma_0(1)$ is $\infty$, we see $M_k^!(1) = M_k^\sharp(1) = S_k^\sharp(1)$, and so Theorem \ref{JenkinsDuke} and Corollary \ref{JenkinsDuke2} may be thought of as proving duality between $\{ \fp 1km (z) \}_m$ and $\{ \gp 1{2 - k}m (z) \}_m$. In collaboration with others, the first author proved duality of this sort for every $N$ of genus zero \cite{Jenkins1, Jenkins2, Jenkins3, Jenkins4, Jenkins6, Jenkins5}.

The driving spirit behind this paper is to extend these duality results and to derive generating functions for the sequences associated to prime levels of nonzero genus. Our first results prove duality between weights congruent to $0$ and $2$ $\pmod{p-1}$ for an infinite class of primes $p$, and duality between arbitrary even weights $k \in \ZZ$ for several small primes. \begin{thm}\label{limitedduality}
	Let $p \not\equiv 1 \pmod{12}$ be prime of genus $g_0 > 0$, and let $k \in 2 \ZZ$ satisfy $k \equiv 0 \pmod{p - 1}$ or $k \equiv 2 \pmod{p - 1}$. Let $\fp pkm (z)$ and $\gp pkm (z)$ be as above. Then for all $m, n \in \ZZ$, we have \[a_k^{(p)}(m,n) = -b_{2-k}^{(p)}(n,m).\]
\end{thm}

\begin{thm}\label{duality}
	Let $p \in \{11, 17, 19, 23, 29, 31, 37\}$, and let $k \in 2 \ZZ$ be arbitrary. Let $\fp pkm = q^{-m} + \sum\limits_{n} a_{k}^{(p)}(m,n) q^n$ and $\gp {p}{k}m (z) = q^{-m} + \sum\limits_{n} b_{k}^{(p)}(m,n) q^n$ be as above. Then for all $m, n \in \ZZ$, we have \[a_k^{(p)}(m,n) = -b_{2-k}^{(p)}(n,m).\]
\end{thm}

The papers \cite{Jenkins1, Jenkins3, Jenkins4} provide explicit formulas for generating functions of canonical bases for levels of genus zero. These results are analogous to Theorem \ref{JenkinsDuke}. In \cite{ElGuindy}, El-Guindy also gave formulas for a variety of generating functions for somewhat different sequences of forms in hyperelliptic levels. Our next result gives generating functions associated with the canonical bases for $M_k^\sharp(p)$ and $S_k^\sharp(p)$ for primes of genus 1. These generating functions are more complicated than in the genus zero case.

Let $v_\infty(f)$ denote the order of vanishing of $f$ at $\infty$ as a function of $q$. Let $F_k^{(p)}(z, \tau) = \sum\limits_m \fp pkm (\tau) q^m$, and let $n_0 = - v_\infty(F_k^{(p)})$ be the index of the first basis element of $M_k^\sharp(p)$. Define $f(z) = \fp p02 (z) = q^{-2} + \sum\limits_{n = -1}^\infty a_n q^n$ and $g(z) = \gp p02 (z) = q^{-2} + \sum\limits_{n = -1}^\infty b_n q^n$.

With this notation, we have the following theorem. \begin{thm}\label{generatingfunctions}
	Let $p \in \{11, 17, 19\}$, so $p$ is a prime level of genus 1. Let $k \in 2 \ZZ$ be arbitrary. Then if there are no gaps in the basis $\{\fp {p}km (\tau)\}_m$, we have \begin{align*}
		F_k^{(p)}(z, \tau) &= \frac{\big( a_{-1} \gp p{2 - k}{n_0 + 1}(z) + \gp p{2-k}{n_0 + 2} (z) \big) \fp pk{- n_0} (\tau) + \gp p{2 - k}{n_0 + 1}(z) \fp pk{- n_0 + 1} (\tau)}{f(z) - f(\tau)} \\
		&= \frac{\big( b_{-1} \fp pk{- n_0} (\tau) + \fp pk{- n_0 + 1} (\tau) \big) \gp p{2 - k}{n_0 + 1}(z) + \fp pk{- n_0} (\tau) \gp p{2-k}{n_0 + 2} (z)}{g(z) - g(\tau)}
	\end{align*} and otherwise \begin{align*}
		F^{(p)}_k(z, \tau) =& \frac{\big( a_1 \gp p{2 - k}{n_0 - 1} (z) + a_{-1} \gp p{2 - k}{n_0 + 1} (z) + \gp p{2 - k}{n_0 + 2} (z) \big) \fp pk{-n_0} (\tau)}{f(z) - f(\tau)} \\
		&+ \frac{a_{-1} \gp p{2 - k}{n_0 - 1} (z) \fp pk{-n_0 + 2} (\tau) + \gp p{2 - k}{n_0 - 1} (z) \fp pk{-n_0 + 3} (\tau)}{f(z) - f(\tau)} \\
		=& \frac{\big( b_1 \fp pk{-n_0} (\tau) + b_{-1} \fp pk{-n_0 + 2} (\tau) + \fp pk{-n_0 + 3} (\tau) \big) \gp p{2 - k}{n_0 - 1} (z)}{g(z) - g(\tau)} \\
		&+ \frac{b_{-1} \fp pk{-n_0} (\tau) \gp p{2 - k}{n_0 + 1} (z) + \fp pk{-n_0} (\tau) \gp p{2 - k}{n_0 + 2} (z)}{g(z) - g(\tau)} .
	\end{align*}
\end{thm}

A natural next step would be to extend these results to general $M_k^\sharp(N)$ and $S_k^\sharp(N)$, as duality appears to hold in levels and weights beyond the scope of this paper (see for instance~\cite{Adams, Kit}). Work on harmonic Maass forms and mock modular forms in \cite{Kane, Bruinier, Griffin} demonstrates a form of duality between Poincar\'{e} series; in addition, the Bruiner-Funke pairing of \cite{Bruinier} is related to the sums that show up in the proofs of Theorems \ref{limitedduality} and \ref{duality}.

In section \ref{Preliminaries} we establish several definitions and lemmas which will be necessary to prove our results. In section \ref{section02} we describe the Fourier expansions of the canonical basis elements in weights congruent to $0 \text{ and } 2 \pmod{p-1}$ and prove Theorem \ref{limitedduality}. In section \ref{OtherWeights} we prove Theorem \ref{duality}, and in section \ref{GeneratingFunctions}, we prove Theorem \ref{generatingfunctions}.

The authors thank Scott Ahlgren, Nick Andersen, Michael Griffin, Ben Kane, and Jeremy Rouse for their invaluable insights.

\section{Preliminaries}\label{Preliminaries}

A weight $k$ weakly holomorphic modular form is a function which is holomorphic on the upper half-plane $\HH = \{ z \in \CC : \Im(z) > 0 \}$, meromorphic at the cusps, and modular of weight $k$ with respect to some congruence group $\Gamma \subseteq \SL_2(\ZZ)$.  Let $\fp Nkm$ and $\gp Nkm$ be the canonical bases for $M_k^\sharp(N)$ and $S_k^\sharp(N)$ described above. These bases are canonical in the sense that given any form $f(z) = \sum a_n q^n \in M_k^\sharp(N)$, we may decompose $f$ as a finite sum $f(z) = \sum\limits_m a_m \fp Nk{-m} (z)$, and a similar decomposition holds for forms in $S_k^\sharp(N)$. Note that this sum is finite, because $\fp Nkn (z)$ is not defined for any $n$ sufficiently negative, and $f \in M_k^\sharp(N)$ has only finitely many terms of negative order.

For $N$ of nonzero genus, $M_0^!(N)$ has no Hauptmodul. In other words, there is no modular function $\chi = q^{-1} + \ldots \in M_0^!(N)$ which generates all of $M_0^!(N)$. Without a Hauptmodul, there may be gaps in the sequence of forms $\{ \fp Nkm \}_m$. Indeed, consider $N = p$ prime with genus $g_0 > 0$. It is clear that $1 = \fp p00 \in M_0^\sharp(p)$, but the next basis element of $M_0^\sharp(p)$ turns out to be $\fp p0{g_0 + 1}$, which demonstrates that gaps in this basis not only exist but can be arbitrarily large. The following lemma implies that the basis $\{ \gp p 2 m \}_m$ also has gaps whenever $p$ is of nonzero genus.

\begin{lemma}\label{zero}
	If $f \in S^\sharp_2(N)$, then $f$ has no constant term.
\end{lemma}

\begin{proof}
	First let $f \in M_2^!(N)$, and let $t$ be a cusp of $\Gamma_0(N)$. We write $f | \gamma_t (\infty) = a_t(0)$ where $\gamma_t \in \SL_2(\ZZ)$ is a matrix mapping $\infty$ to $t$. Although the Fourier expansion of $f$ at $t$ is not well-defined, its constant term $a_t(0)$ is. We claim that $\sum\limits_{t} a_t(0) = 0$, where the sum is over the cusps of $\Gamma_0(N)$. If so, the result is immediate, for if $f \in S_2^\sharp(N)$, then whenever $t \neq \infty$ we have $a_t(0) = 0$. Thus $0 = \sum\limits_t a_t(0) = a_\infty(0)$.

	By Theorem 3.7 of \cite{Miranda}, if $\omega$ is a differential 1-form on a compact Riemann surface $X$, then \[\sum\limits_{p \in X} \Res_p(\omega) = 0.\] Set $\omega = f(z)dz$ with $f \in M^!_2(N)$. For $t$ a cusp of $f$, $\Res_t(f) = \frac{a_t(0)}{2 \pi i}$. As $f$ is weakly holomorphic as a 1-form, we see \[\frac{1}{2 \pi i} \sum\limits_{t} a_t(0) = 0\] and the result follows.
\end{proof}

From the lemma, there is no $\gp N20 \in S_2^\sharp(N)$. But if $N$ is of genus $g_0 > 0$, then $\gp N2m$ exists for some $m > 0$ and for some $m < 0$, and hence there is a gap in the sequence $\{\gp N2m\}_m$ of canonical basis elements.

For $N \in \NN$ and $k$ an even integer, let $\miss[S](N, k)$ be the set of all integers $i$ such that there is no form $f \in S_k(N)$ with $v_\infty(f) = i$ but there do exist forms $g, h \in S_k(N)$ with $v_\infty(g) < i < v_\infty(h)$. Informally, $\miss[S](N, k)$ collects the indices that are skipped in the row-reduced basis of $S_k(N)$. We define $\miss[M](N, k),$ $\misss[S](N, k),$ $\misss[M](N, k)$ analogously. We also set $\cms[S](N, k) = \abs{\miss[S](N, k)}$, and define $\cms[M](N, k), \cmss[S](N,k), \cmss[M](N, k)$ analogously.

From now on, we restrict our attention to $N = p$ prime. If $p > 3$, we may define $\la(z) = \frac{\eta(pz)^{2p}}{\eta(z)^2}$. By Newman \cite{Newman1, Newman2} (see also \cite{Ligozat}), $\la$ is a weight $p-1$ modular form for $\Gamma_0(p)$ with all its zeros at $\infty$ and with $v_\infty (\la) = \frac{p^2-1}{12}$. Note that $f \mapsto \la^\ell f$ establishes a bijection from $M^\sharp_k(p)$ to $M^\sharp_{k + \ell (p-1)}(p)$ and from $S_k^\sharp(p)$ to $S_{k + \ell (p-1)}^\sharp(p)$. Thus the sets $\misss[S](p,2), \ldots \misss[S](p,p-1)$ control the behavior of $\misss[S] (p, k)$ and $\misss[M] (p, k)$ for general weight $k$.

The following results are classical (see Exercise 3.1.4 and Theorem 3.5.1 of \cite{Diamond}, for example), but essential for the computations that follow.

\begin{lemma}\label{genus}
	Let $p$ be prime. Then the genus of $\Gamma_0(p)$ is
	\begin{align*}
		g_0 = g_0(p) = \begin{cases}
			\floor{\frac{p+1}{12}} - 1 & \text{if } p \equiv 1 \pmod{12}, \\
			\floor{\frac{p+1}{12}} & \text{otherwise}.
		\end{cases}
	\end{align*}
\end{lemma}

Let $\mathcal{E}_{k}(\Gamma_0(N)) = \mathcal{E}_{k}(N)$ be the space of weight $k$ Eisenstein forms for $\Gamma_0(N)$.

\begin{lemma} \label{dimensions}
	Let $p > 3$ be prime. Then $\dim S_2(p) = g_0$, and for all $k > 2$, we have \begin{align*}
		\dim S_{k}(p) &=
		\begin{cases}
			g_0 k  - (g_0 + 1) + 2\floor{\frac{k}{3}} + 2\floor{\frac{k}{4}} & \text{if } p \equiv 1 \pmod{12} \\
			g_0 k - (g_0 + 1) & \text{if } p \equiv -1 \pmod{12} \\
			g_0 k - (g_0 + 1) + 2\floor{\frac{k}{4}} & \text{if }  p \equiv 5 \pmod{12} \\
			g_0 k - (g_0 + 1) + 2\floor{\frac{k}{3}} & \text{if } p \equiv -5 \pmod{12}
		\end{cases}
	\end{align*}
	
	Also, $\dim \mathcal{E}_2(p) = 1$ and $\dim \mathcal{E}_k(p) = 2$ for $k > 2$.
\end{lemma}

We also require the following lemma, which Ogg attributes to Atkin in \cite{Ogg}.

\begin{lemma}\label{Ogg}
	For $p$ prime, $\infty$ is not a Weierstrass point of $X_0(p)$. Equivalently, $\miss[S] (p, 2) = \emptyset$ and $\cms[S] (p, 2) = 0$.
\end{lemma}

For $p$ prime and $k$ an even natural number, let \begin{align*}
		s_{p, k} &= \max\limits_{0 \neq f \in S_k(N) } v_\infty(f), \\
		m_{p, k} &= \max\limits_{0 \neq f \in M_k(N) } v_\infty(f). \end{align*} Note that $E_{k,p}(z) = \frac{p E_k(pz) - E_k(z)}{p - 1} = 1 + \ldots \in M_k(p)$, so $m \in \miss[M](p, k)$ if and only if $0 < m < m_{p,k}$ and there are no forms in $M_k(p)$ with leading term $q^m$. The number of such forms in the reduced row echelon basis is $\cms[M](p, k) = m_{p, k} - \dim M_k(p)$. Similarly, the existence of eigenforms (which must be $O(q)$ but not $O(q^2)$) shows that $m \in \miss[S](p, k)$ if and only if $1 < m < s_{p, k}$ and there are no forms in $S_k(p)$ with leading term $q^m$, and the number of such forms is $\cms[S](p, k) = s_{p, k} - \dim S_k(p)$. Thus for all $p$ prime, $k > 0$ even, we have \begin{align*}
	\cms[M] (p, k) &= m_{p, k} - \dim M_k(p), \\
	\cms[S] (p, k) &= s_{p, k} - \dim S_k(p).
\end{align*}

A theorem of Ahlgren, Masri, and Rouse allows us to control $\cms[S] (p, k)$ for $2 \leq k \leq p - 1$.

\begin{thm}[\cite{Ahlgren}, Theorem 1.4]\label{AhlgrensTheorem}
	Suppose that $k \geq 2$, and that $p \geq \max \{5, k+1\}$ is prime. Then we have \[s_{p,k} \leq \frac{k p}{12} - \frac{1}{2} \alpha_2(1,k p) - \frac{1}{3} \alpha_3(1,k p)\] where $\alpha_2(1, k p)$ and $\alpha_3(1, k p)$ are the orders of vanishing for $i$ and $e^{2 \pi i /3}$ forced by valence considerations in $M_{k p}(1)$.
\end{thm}

Theorem 1.4 of \cite{Ahlgren} is more general than the above claim, but this formulation is sufficient for our needs. A computation shows that for $p > 3$ prime, $2 \leq k \leq p - 1$ even, we have $\cms[S] (p, k) \leq \floor{\frac{p-k}{12}} + 1 + \epsilon_{p,k}$ where
	\[
	\epsilon_{p,k} =
	\begin{cases}
		1 & p - k + 1 \equiv 0 \pmod{12}, \\
		-1 & p - k - 1 \equiv 0 \pmod{12}, \\
		0 & \text{otherwise.}
	\end{cases}
	\] In particular, letting $k = p - 1$, we see that for $p > 3$ prime, $p \not\equiv 1 \pmod{12}$, we have $\cms (p, p-1) = 0$. In other words, there are no gaps in the basis for $S_{p - 1}(p)$.	If $p \equiv 1 \pmod{12}$, we only have $\cms (p, p-1) \leq 1$; this bound is not tight enough to prove our results.

\section{Weights 0 and 2}\label{section02}

The results in the previous section allow us to show that for $p \not\equiv 1 \pmod{12}$ a prime of nonzero genus, there are no unexpected gaps in the bases $\{ \fp pkm \}_m$ and $\{ \gp pkm \}_m$ in weights $k \equiv 0, 2 \pmod{p - 1}$.

\begin{prop}\label{0dim}
	Define $\lambda_p = \frac{p^2 -1}{12} = v_\infty(\la)$. Let $p \not\equiv 1 \pmod{12}$ be prime of genus $g_0 > 0$. Then $\fp p0m$ is defined for $m = 0$ and for \[m \geq \lambda_p - s_{p,p-1} = \lambda_p - \dim S_{p-1}(p) = g_0(p) + 1\] and is of the form \[\fp p0m = q^{-m} + a_0^{(p)}(m,-g_0(p)) q^{-g_0(p)} + \ldots + a_0^{(p)}(m,-1) q^{-1} + a_0^{(p)}(m,1) q + \ldots \text{ .}\] Likewise, $\gp p0m$ is defined for $m \geq g_0(p) + 1$ and is of the form \[\gp p0m = q^{-m} + b_0^{(p)}(m, -g_0(p)) q^{-g_0(p)} + \ldots \text{ .}\]
\end{prop}
We remark that $a_0^{(p)}(m,n) = b_0^{(p)}(m,n)$ except when $n = 0$.
\begin{proof}
	Let $\{ \gp p{p-1}m = q^{-m} + \ldots \}_{m = -\dim S_{p-1}(p)}^{-1}$ be a row-reduced basis for $S_{p-1}(p)$; the basis has this form by the argument above. The Eisenstein series $E_{p - 1}$ is an oldform of $M_{p-1}(p)$ and has constant term, unlike all the $\gp p{p-1}i$ previously described, and $\la$ is a weight $p-1$ form with higher order of vanishing than any of the forms in $S_{p-1}(p)$. Then as $\dim M_k(p) = \dim S_k(p) + 2$, these forms together form a basis for $M_{p-1}(p)$. Start with $\la$, the form of maximal order of vanishing in this basis, and begin dividing by $\la$. Note that $\la / \la = 1$ is in reduced form, and no nonzero element of $M_0^\sharp(p)$ has positive order of vanishing at infinity, else multiplying by $\la$ we would have a form in $M_{p-1}(p)$ with order of vanishing at infinity greater than $\la$, which is impossible. Now consider $\gp p{p-1}{s_{p,p-1}}$, the next form in our basis for $M_k(p)$, and divide by $\la$. This yields a nonconstant form in $S_0^\sharp(p) \subseteq M_0^\sharp(p)$ with a pole at infinity of minimal order. In particular, $v_\infty(\frac{\gp p{p-1}{s_{p,p-1}}}{\la}) = - g_0(p) - 1$. Subtracting off an appropriate multiple of $1$ from this form, we have obtained $\fp p0{g_0(p) + 1}$. Continuing in this manner and row-reducing by previously constructed forms, we obtain $1, \fp p0{g_0(p) + 1}, \ldots, \fp p0{\lambda_p}$. Now suppose we have forms $\fp p0m$ for any $g_0(p) + 1 \leq m \leq n$ with $n \geq \lambda_p$. Then (noting that $g_0(p) + 1 \leq n - g_0(p) \leq n$) the form $\fp p0{n - g_0(p)} \fp p0{g_0(p) + 1} = q^{-n - 1} + \ldots$ is in $M_0^\sharp(p)$, and row-reducing with previously constructed forms, we obtain $\fp p0{n+1}$. Continuing inductively, we obtain exactly the $\fp p0m$ claimed. The construction of $\gp p0m$ is analogous.
\end{proof}

The corollary below is immediate by writing $M^\sharp_{\ell (p -1)}(p) = \la^\ell M^\sharp_0(p)$.

\begin{cor}\label{dim0}
	Let $p \not\equiv 1 \pmod{12}$ be prime of genus $g_0 > 0$. Let $k = \ell (p-1)$, with $\ell \in \ZZ$. The basis elements $\fp p{\ell (p-1)}m$ exist for $m = - \ell \lambda_p$ and for $m \geq g_0(p) - \ell \lambda_p + 1$ and are of the form \begin{align*}
		\fp p{\ell (p-1)}m = q^{-m} &+ a_{\ell (p-1)}^{(p)}(m, -g_0(p) + \ell \lambda_p) q^{-g_0(p) + \ell \lambda_p} + \ldots \\
		&+ a_{\ell (p-1)}^{(p)}(m, \ell \lambda_p - 1) q^{\ell \lambda_p - 1} + a_{\ell (p-1)}^{(p)}(m,\ell \lambda_p + 1) q^{\ell \lambda_p + 1} + \ldots \text{ .}
	\end{align*}
	
	Likewise, $\gp p{\ell (p-1)}m$ is defined for $m \geq g_0(p) - \ell \lambda_p + 1$ and is of the form \[\gp p{\ell (p-1)}m = q^{-m} + b_0^{(p)}(m, - g_0(p) + \ell \lambda_p) q^{-g_0(p) + \ell \lambda_p} + \ldots \text{ .}\]
	\end{cor}

We obtain similar results for weight $2$.

\begin{prop}
	Let $p \not\equiv 1 \pmod{12}$ be prime of genus $g_0 > 0$. Then $\fp p2m$ is defined for $m \geq - g_0(p)$ and is of the form \[\fp p2m = q^{-m} + a_2^{(p)}(m, g_0(p)+1) q^{g_0(p)+1} + \ldots \\ \text{ .}\] Likewise, $\gp p2m$ is defined for $m \geq - g_0(p)$, $m \neq 0$ and is of the form \[\gp p2m = q^{-m} + b_2^{(p)}(m, g_0(p)+1) q^{g_0(p) + 1} + \ldots \\ \text{ .}\]
\end{prop}

We note that $\fp p20 - E_{2, p}$ is a cusp form, accounting for the $m=0$ exception in the basis $\{\gp p2m\}$.

\begin{proof}
	Lemma \ref{Ogg} gives us $\fp p2m$ for $-g_0(p) \leq m \leq -1$. As $M_2(p)$ has an Eisenstein form with constant term,  namely $E_{2,p}(z)$, we also have $\fp p20 (z)$. Now suppose we have $\fp p2m$ whenever $-g_0(p) \leq m \leq n$ with $n \geq 0$. Then $n - g_0(p) \geq - g_0(p)$ and $\fp p0{g_0(p) + 1} \fp p2{n - g_0(p)} = q^{-n-1} + \ldots$ is a weight 2 weakly holomorphic form, and row-reducing with previously constructed forms we obtain $\fp p2{n+1}$. Continuing inductively, we obtain exactly the $\fp p2m$ we claimed.
	
	The construction of $\gp p2m$ is analogous, with the slight complication that $\gp 2p0$ does not exist (Lemma \ref{zero}), and we must construct $\gp p2{g_0(p)+1}$ by multiplying $\gp p2{-1}$ and $\fp p0{g_0(p)+2}$.
\end{proof}

Now writing $M^\sharp_{2 + \ell (p - 1)}(p) = \la^\ell M^\sharp_2(p)$, we obtain the following corollary.

\begin{cor}\label{dim2}
	Let $p \not\equiv 1 \pmod{12}$ be prime of genus $g_0 > 0$. Let $k = 2 + \ell (p-1)$, with $\ell  \in \ZZ$. Then $\fp p{2+ \ell \lambda_p}m$ is defined for $m \geq - g_0(p) - \ell \lambda_p$ and is of the form \[\fp p{2 + \ell \lambda_p}m = q^{-m} + a_{2 + \ell \lambda_p}^{(p)}(m, g_0(p) + \ell \lambda_p + 1) q^{g_0(p)  + \ell \lambda_p + 1} + \ldots \text{ .}\]
	
	Likewise, $\gp p{2 + \ell \lambda_p}m$ is defined for $m \geq - g_0(p) - \ell \lambda_p$, $m \neq - \ell \lambda_p$ and is of the form \[\gp p{2 + \ell \lambda_p}m = q^{-m} + b_{2 + \ell \lambda_p}^{(p)}(m, \ell \lambda_p) q^{\ell \lambda_p} + b_{2 + \ell \lambda_p}^{(p)}(m,g_0(p)  + \ell \lambda_p + 1) q^{g_0(p)  + \ell \lambda_p + 1} + \ldots\]	with $b_{2 + \ell \lambda_p}^{(p)}(m, \ell \lambda_p) = 0$ for $i < - \ell \lambda_p$.
\end{cor}

With the above machinery in place, it is now straightforward to prove Theorem \ref{limitedduality}, which gives duality for all $k \equiv 0, 2 \pmod{p - 1}$.

\begin{thm}\label{limittedduality2}
	Let $p \not\equiv 1 \pmod{12}$ be prime of genus $g_0 > 0$, and let $k \in 2 \ZZ$ satisfy $k \equiv 0 \pmod{p - 1}$ or $k \equiv 2 \pmod{p - 1}$. Then for all $m, n \in \ZZ$, we have \[a_k^{(p)}(m,n) = -b_{2-k}^{(p)}(n,m).\]
\end{thm}

\begin{proof}
	First suppose that $k \equiv 0 \pmod{p - 1}$. Write $k = \ell (p-1)$. Clearly the product $\fp p{\ell (p-1)}m \gp p{2- \ell (p-1)}n$ is an element of $S_2^\sharp(p)$, so by Lemma \ref{zero} it has no constant term. On the other hand, by inspection of Fourier series, its constant term is \[a_{\ell (p-1)}^{(p)}(m,n) + b_{2 - \ell (p-1)}{(p)}(n,m) + \sum\limits_{i} a_{l(p-1)}^{(p)}(m,i)b_{2- \ell (p-1)}^{(p)}(n,-i).\]
	
	If $i \leq g_0 - \ell \lambda_p$, then  $a_{\ell (p-1)}^{(p)}(m,i) = 0$. If $i > g_0 - \ell \lambda_p$ and $i \neq \ell \lambda_p$, then $b_{2 - \ell (p-1)}^{(p)}(n,-i) = 0.$ Finally, for $i = \ell \lambda_p$, we have $a_{\ell (p-1)}^{(p)}(m, \ell \lambda_p) = 0$. Then  $\sum\limits_{i} a_{\ell (p-1)}^{(p)}(m, i) b_{2- \ell (p-1)}^{(p)}(n, -i) = 0$ and \[a_{\ell (p-1)}^{(p)}(m, n) = -b_{2- \ell (p-1)}^{(p)}(n, m)\] as desired. The case $k \equiv 2 \pmod{p - 1}$ is analogous.
\end{proof}

The sum $\sum\limits_{i} a_{l(p-1)}^{(p)}(m,i)b_{2- \ell (p-1)}^{(p)}(n,-i)$ may also be interpreted as the Petersson scalar product in the Bruiner-Funke pairing of \cite{Bruinier}.

\section{Other Weights}\label{OtherWeights}

The proof above reveals that the reason Zagier duality holds is that the gaps in the bases $\{ \fp pkm \}_m$ and $\{ \gp p{2 - k}n \}_n$ line up properly. When this happens, verifying duality becomes a straightforward computation. It is clear that $\miss[M](p, k) \subseteq \misss[M](p, k)$, and $\miss[S](p, k) \subseteq \misss[S](p, k)$. In fact, these inclusions are generally equalities.

\begin{prop}
	Let $p$ be a prime of genus $g_0 > 0$, with $p \not\equiv 1 \pmod{12}$, and let $k$ be an even integer with $2 < k \leq p - 1$. Then \begin{align*}
	\misss(p, k) &= \miss (p, k) \quad \quad \text{ and} \\
	\misss[M](p, k) &=
	\begin{cases}
	\miss[M](p, k) & \text{ if } k < p - 1, \\
	\miss[M](p, k) \cup \{ 0 \} & \text{ if } k = p - 1. \\
	\end{cases}
	\end{align*} Equivalently, the elements $\{\gp pkm \}$ exist if they are in $S_k(p)$ or if they have nonpositive index, and the elements $\{\fp pkm \}$ exist if they are in $M_k(p)$ or if they have negative index.
\end{prop}

\begin{proof}
	We first establish a lazy lower bound for $\dim M_k(p)$. Writing $g_0(p) = g_0$, we see $\dim M_k (p) = \dim S_k(p) + \dim \mathcal{E}_k(p) \geq g_0 + 2$. Then there exists $\fp pk{-n}$ with $n \geq g_0 + 1$. We know $\gp p0n$ exists by Proposition \ref{0dim} and so $\fp pk{-n} \gp p0n = 1 + \ldots $ is a weight $k$ form that vanishes at 0, and so row-reducing by the elements of $S_k(p)$, we obtain $\gp pk0$. We may now construct $\{\gp pkm\}$ for $m \geq 0$ inductively by considering the elements $\fp pk{-n} \gp p0{m + n}$ sequentially and then row-reducing.
	
	An analogous argument holds for $\misss[M](p, k)$.
\end{proof}

Now for $k = k' + \ell (p - 1)$ with $2 < k' < p - 1$, if we are given $\miss[M](p, k')$, it is easy to see that $\{ \fp p{k'}m \}_m$ are defined exactly for $m \geq - \dim M_{k'}(p) - \cms[M] (p, k') + 1$ with $m \not\in -\miss[M](p, k')$. Then we may take $\{ \la^\ell \fp p{k'}m \}_i$ and row-reduce to obtain $\{ \fp pkm \}_m$. Making a similar argument for $\{ \gp pkm \}_m$, we obtain the following corollary.

\begin{cor}\label{dimnot02}
	Let $p$ be a prime of genus $g_0 > 0$, with $p \not\equiv 1 \pmod{12}$, and let $k \in 2 \ZZ$ be arbitrary with $k \not\equiv 0, 2 \pmod{p - 1}$. Write $k = k' + \ell (p-1)$, with $k' \in \{4, \ldots, p - 3\}$. Then $\misss (p, k) = \misss (p, k') + \ell \lambda_p$ under pointwise addition, and $\fp pkm$ exists for $m \geq -\dim M_{k'}(p) - \cms[M] (p, k') - \ell \lambda_p + 1$ with $m \not\in -\miss[M](p, k') - \ell \lambda_p$ and is of the form \begin{align*}
		\fp pkm = q^{-m} &+ \sum\limits_{n \in \miss[M](p, k)} a_k^{(p)}(m, n + \ell \lambda_p) q^{n + \ell \lambda_p} \\
		&+ a_k^{(p)}(m ,\dim M_{k'}(p) + \cms[M] (p, k) + \ell \lambda_p) q^{\dim M_{k'}(p) + \cms[M] (p, k) + \ell \lambda_p} + \ldots \text{ .} \\
	\end{align*}
	
	Likewise, $\gp pkm$ exists for $m \geq - \dim S_{k'}(p) - \cms (p, k') - \ell \lambda_p$ with $m \not\in - \miss (p, k') - \ell \lambda_p$ and is of the form \begin{align*}
		\gp pkm = q^{-m} &+ \sum\limits_{n \in \miss (p, k)} b_k^{(p)}(m, n + \ell \lambda_p) q^{n + \ell \lambda_p} \\
		&+ b_k^{(p)}(m, \dim S_{k'}(p) + \cms (p, k) + \ell \lambda_p + 1) q^{\dim S_{k'}(p) + \cms (p, k) + \ell \lambda_p + 1} + \ldots \text{ .}
	\end{align*}
\end{cor}

It is straightforward to compute $\miss[M](p, k')$ and $\miss[S](p, k')$ for $0 < k' < p - 1$ using a computer algebra system (the authors used Sage), which allows us to verify duality directly in a number of cases.

We may now prove Theorem \ref{duality}, which we restate here.

\begin{thm}\label{duality2}
	Let $p \in \{11, 17, 19, 23, 29, 31, 37\}$, and let $k \in 2 \ZZ$ be arbitrary. Let $\fp pkm = q^{-m} + \sum\limits_{n} a_{k}^{(p)}(m,n) q^n$ and $\gp {p}{k}m (z) = q^{-m} + \sum\limits_{n} b_{k}^{(p)}(m,n) q^n$ be as above. Then for all $m, n \in \ZZ$, we have \[a_k^{(p)}(m,n) = -b_{2-k}^{(p)}(n,m).\]
\end{thm}

\begin{proof}
	We prove only the case $p = 17$ by way of illustration, but provide enough data for readers to work out the remaining cases.
	
	Let $p = 17$. Note that $g_0(17) = 1$ and that $\lambda_{17} = 24$. It is also clear that $17 \not\equiv 1 \pmod{12}$. We also see that for $k > 2$ even, we have $\dim S_{k}(17) = k + 2 \floor{\frac{k}{4}} - 2$, and $\dim M_k(17) = k + 2 \floor{\frac{k}{4}}$.
	
	A computation reveals that \[\miss[M](17, 6) = \{ 7 \}, \quad \miss[S](17, 12) = \{ 16 \}.\] For all other $4 \leq k \leq 14 = 17 - 3$, we have \[\cms[S](17, k) = \cms[M](17, k) = 0.\] If $k \equiv 0, 2 \pmod{16}$, the result was proven in Theorem \ref{limittedduality2} above. So suppose now that $k \not\equiv 0, 2 \pmod{16}$ and write $k = k' + 16 \ell$ with $k' \in \{4, 6, \ldots, 14 \}$.
		
	First suppose $k' \neq 6$. By Corollary \ref{dimnot02}, we have $\fp {17}km$ defined for $m \geq -k' - 2 \left\lfloor \frac{k'}{4} \right\rfloor - 24 \ell + 1$, and $\fp {17}km$ is of the form \[\fp {17}km = q^{-m} + a_k^{(17)}\left( m, k' + 2 \left\lfloor \frac{k'}{4} \right\rfloor + 24 \ell \right) q^{k' + 2 \left\lfloor \frac{k'}{4} \right\rfloor + 24 \ell } + \ldots \text{ .}\]
	
	Likewise, a computation shows that $\gp {17}{2 - k}m$ is defined for $m \geq k' + 2 \left\lfloor \frac{k'}{4} \right\rfloor + 24 \ell $. We see $\gp {17}{2 - k}m$ is of the form \[\gp {17}{2 - k}m = q^{-m} + b_{2 - k}^{(17)} \left(m, - k' - 2 \left\lfloor \frac{k'}{4} \right\rfloor - 24 \ell + 1 \right) q^{- k' - 2 \left\lfloor \frac{k'}{4} \right\rfloor - 24 \ell + 1} + \ldots \text{ .}\]
		
	Now let $\fp {17}km \in M_k^\sharp(17)$, $\gp {17}{2-k}n \in S_{2 - k}^\sharp(17)$ be arbitrary. Then $\fp {17}km \gp {17}{2-k}n$ is an element of $S_2^\sharp(17)$. As such, by Lemma \ref{zero}, it has no constant term. On the other hand, by inspection of Fourier series, its constant term is \[a_{k}^{(17)}(m, n) + b_{2 - k}^{(17)}(n, m) + \sum\limits_i a_{k}^{(17)}(m, i) b_{2 - k}^{(17)}(n, -i).\]
	
	But for $i < k' + 2 \left\lfloor \frac{k'}{4} \right\rfloor + 24 l$ we see that $a_k^{(17)}(m, i) = 0$, and on the other hand if $i \geq k' + 2 \left\lfloor \frac{k'}{4} \right\rfloor + 24 \ell $ then $b_{2 - k}^{(17)}(n, -i) = 0$. Duality follows.
	
	Now let $k' = 6$. By Corollary \ref{dimnot02}, we have $\fp {17}km$ defined for $i \geq -8 - 24 \ell $ with $m \neq - 7 - 24 \ell $, and $\fp {17}km$ is of the form \[\fp {17}km = q^{-m} + a_k^{(17)}(m, 7 + 24 \ell ) q^{7 + 24 \ell } + a_k^{(17)}(m, 9 + 24 \ell ) q^{9 + 24 \ell } + \ldots \text{ .}\] Likewise, $\gp {17}{2 - k}m$ is defined for $m \geq 7 + 24 \ell $ with $m \neq 8 + 24 \ell $, and is of the form \[\gp {17}{2 - k}m = q^{-m} + b_{2 - k}^{(17)}(m, -8 - 24 \ell ) q^{-8 - 24 \ell } + b_{2 - k}^{(17)}(m, -6 - 24 \ell ) q^{-6 - 24 \ell } + \ldots \text{ .}\]

	Finally, let $\fp {17}km \in M_k^\sharp(17)$, $\gp {17}{2-k}n \in S_{2 - k}^\sharp(17)$ be arbitrary (with $k' = 6$). Then $\fp {17}km \gp {17}{2-k}n$ is an element of $S_2^\sharp(17)$. As such, by Lemma \ref{zero}, it has no constant term. On the other hand, by inspection of Fourier series, its constant term is \[a_{k}^{(17)}(m,n) + b_{2 - k}{(17)}(n,m) + \sum\limits_i a_{k}^{(17)}(m,i) b_{2 - k}^{(17)}(n, -i).\] If $i < 7 + 24 \ell $, then  $a_{k}^{(17)}(m , i) = 0$. If $m \geq 7 + 24 \ell $ and $m \neq 8 + 24 \ell $, then $b_{2 - k}^{(17)}(n , -i) = 0$. Finally, for $m = 8 + 24 \ell $, we have $a_{k}^{(17)}(m, 8 + 24 \ell ) = 0$, and we have duality for $p = 17$.

	It turns out that for $p = 11$, for $4 \leq k' \leq 8$, we have $\cms[S](11, k') = \cms[M](11, k') = 0$. For $17 \leq p \leq 37$ prime, all $4 \leq k' \leq p - 3$ yield $\cms[S](p, k') = \cms[M](p, k') = 0$, with the following exceptions:

	\begin{center}
		\bgroup
		\def\arraystretch{1.5}
		\begin{tabular}{ |l|l|l| }
			\hline
			$p$ prime & $\miss[M](p, \cdot)$ & $\miss[S](p, \cdot)$ \\
			\hline
			$p = 17$ & $\miss[M](17, 6) = \{ 7 \}$ & $\miss[S](17, 12) = \{ 16 \}$ \\
			\hline
			\multirow{2}{3em}{$p = 19$} & $\miss[M](19, 4) = \{ 5 \}$ & $\miss[S](19, 16) = \{ 24 \}$ \\
					& $\miss[M](19, 8) = \{ 11 \}$ & $\miss[S](19, 12) = \{ 18 \}$ \\
			\hline
			$p = 23$ & $\miss[S](23, 12) = \{ 21, 22 \}$ & $\miss[M](23, 12) = \{ 21 \}$ \\
			\hline
			\multirow{3}{3em}{$p = 29$} & $\miss[M](29, 6) = \{ 12, 13 \}$ & $\miss[S](29, 24) = \{ 56 \}$ \\
					& $\miss[M](29, 12) = \{ 27 \}$ & $\miss[S](29, 18) = \{ 41, 42 \}$ \\
					& $\miss[M](29, 18) = \{ 41 \}$ & $\miss[S](29, 12) = \{ 27, 28 \}$ \\
			\hline
			\multirow{5}{3em}{$p = 31$} & $\miss[M](31, 4) = \{ 8, 9 \}$ & $\miss[S](31, 28) = \{ 70 \}$ \\
					& $\miss[M](31, 8) = \{ 18, 19 \}$ & $\miss[S](31, 24) = \{ 60 \}$ \\
					& $\miss[M](31, 12) = \{ 29 \}$ & $\miss[S](31, 20) = \{ 49, 50 \}$ \\
					& $\miss[M](31, 16) = \{ 39 \}$ & $\miss[S](31, 16) = \{ 39, 40 \}$ \\
					& $\miss[M](31, 20) = \{ 49 \}$ & $\miss[S](31, 12) = \{ 29, 30 \}$ \\
			\hline
			\multirow{3}{3em}{$p = 37$} & $\miss[M](37, 12) = \{ 36 \}$ & $\miss[S](37, 26) = \{ 77 \}$ \\
					& $\miss[M](37, 14) = \{ 40 \}$ & $\miss[S](37, 24) = \{ 73 \}$ \\
					& $\miss[M](37, 26) = \{ 77 \}$ & $\miss[S](37, 12) = \{ 35, 36 \}$ \\
			\hline
		\end{tabular}
		\egroup
	\end{center}
	
	The prime $p = 37$ deserves special attention here, since $p \equiv 1 \pmod{12}$ and so the results proven in the previous section do not immediately apply to it. However, another computation reveals that $\cms[S](37, 36) = \cms[M](37, 36) = 0$. We see that the arguments in sections \ref{section02} and \ref{OtherWeights} go through without impediment.
\end{proof}

\section{Generating Functions}\label{GeneratingFunctions}

Let $F_k^{(p)}(z, \tau)$ be the generating function for $\{\fp pkm\}_m$ defined in the introduction, and let $G_k^{(p)}(z, \tau)$ be the corresponding generating function for $\{\gp pkm\}_m$. If $p$ is of nonzero genus and $p \not\equiv 1 \pmod{12}$, and $k \equiv 0 \pmod{p - 1}$ or $k \equiv 2 \pmod{p - 1}$, then $F_k^{(p)}(z, \tau) = - G_{2 - k}^{(p)}(z, \tau)$ by Theorem \ref{limitedduality}. If $p \in \{11, 17, 19, 23, 29, 31, 37 \}$ then $F_k^{(p)}(z, \tau) = - G_{2 - k}^{(p)}(z, \tau)$ for any even integer $k$ by Theorem \ref{duality}.

There is a recurrence relation for $\{ \fp pki \}$ and for $\{ \gp pki \}$ in terms of $\fp p0{g_0 + 1}$ and previously constructed $\{ \fp pki \}$ or $\{ \gp pki \}$. This recurrence relation can be used to construct explicit formulas for $F_k^{(p)}(z, \tau)$ and $G_k^{(p)}(z, \tau)$. Theorem \ref{generatingfunctions} is an application of these techniques in genus 1; we restate it here with the notations we have established.

Let $F_k^{(p)}(z, \tau) = \sum\limits_m \fp {p}km (\tau) q^m$, and let $n_0 = - v_\infty (F_k^{(p)})$ be the index of the first basis element of $M_k^\sharp(p)$. Define $f(z) = \fp p02 (z) = q^{-2} + \sum\limits_{n = -1}^\infty a_n q^n$ and $g(z) = \gp p02 (z) = q^{-2} + \sum\limits_{n = -1}^\infty b_n q^n$.

\begin{thm}\label{generatingfunctions2}
	Let $p \in \{11, 17, 19\}$, and let $k \in 2 \ZZ$ be arbitrary. If $\cmss[M] (p, k) = 0$, we have \begin{align*}
		F_k^{(p)}(z, \tau) &= \frac{\big( a_{-1} \gp p{2 - k}{n_0 + 1}(z) + \gp p{2-k}{n_0 + 2} (z) \big) \fp pk{- n_0} (\tau) + \gp p{2 - k}{n_0 + 1}(z) \fp pk{- n_0 + 1} (\tau)}{f(z) - f(\tau)} \\
		&= \frac{\big( b_{-1} \fp pk{- n_0} (\tau) + \fp pk{- n_0 + 1} (\tau) \big) \gp p{2 - k}{n_0 + 1}(z) + \fp pk{- n_0} (\tau) \gp p{2-k}{n_0 + 2} (z)}{g(z) - g(\tau)}
	\end{align*} and otherwise \begin{align*}
		F^{(p)}_k(z, \tau) =& \frac{\big( a_1 \gp p{2 - k}{n_0 - 1} (z) + a_{-1} \gp p{2 - k}{n_0 + 1} (z) + \gp p{2 - k}{n_0 + 2} (z) \big) \fp pk{-n_0} (\tau)}{f(z) - f(\tau)} \\
		&+ \frac{a_{-1} \gp p{2 - k}{n_0 - 1} (z) \fp pk{-n_0 + 2} (\tau) + \gp p{2 - k}{n_0 - 1} (z) \fp pk{-n_0 + 3} (\tau)}{f(z) - f(\tau)} \\
		=& \frac{\big( b_1 \fp pk{-n_0} (\tau) + b_{-1} \fp pk{-n_0 + 2} (\tau) + \fp pk{-n_0 + 3} (\tau) \big) \gp p{2 - k}{n_0 - 1} (z)}{g(z) - g(\tau)} \\
		&+ \frac{b_{-1} \fp pk{-n_0} (\tau) \gp p{2 - k}{n_0 + 1} (z) + \fp pk{-n_0} (\tau) \gp p{2 - k}{n_0 + 2} (z)}{g(z) - g(\tau)} .\\
	\end{align*}
\end{thm}

\begin{proof}
	We prove the second of these two identities. The proof for the first expression is similar but simpler.
	
	Suppose $\cmss (p, k) \neq 0$. Then by inspection, we have either $k \equiv 0 \pmod{p - 1}$, or $p = 17$ and $k \equiv 6 \pmod{16}$, or $p = 19$ and $k \equiv 4, 8 \pmod{18}$. In any of these cases, $\cmss[M] (p, k) = 1$. We write $\misss[M](p, k) = \{ \eta \}$.
	
	For simplicity, we suppress dependence on $p$, denoting $\fp pki = f_{k,i}$ and $\gp pki = g_{k,i}$, as well as $F_k^{(p)} = F_k$. Write $k = k' + (p - 1) \ell$ with $0 \leq k' < p - 1$. By definition, $n_0 = v_\infty(F_k)$; by inspection we see $\eta = n_0 - 1$. Then $f_{k, -n_0}$ is defined for $n \geq -n_0$ when $n \neq -n_0 + 1$. For $n \geq -n_0$, $f_{k, n}$ may be written in the form $f_{k, n} = q^{-n} + a_k(n, n_0 - 1)q^{n_0 - 1} + a_k(n, n_0 + 1)q^{n_0 + 1} + \ldots$. Note that if $n = -n_0$, then $a_{k}(-n_0, n_0 - 1) = 0$.
	
	We see that \begin{align*}
		f(z) f_{k, -n_0}(z) =& (q^{-2} + a_{-1} q^{-1} + a_1 q + \ldots)(q^{n_0} + a_{k}(-n_0 , n_0 + 1) q^{n_0 + 1} + \ldots) \\
		=& q^{n_0 - 2} + C q^{n_0 - 1} + \big( a_{k}(-n_0 , n_0 + 2) + a_{-1} a_{k}(-n_0 , n_0 + 1) \big) q^{n_0}
	\end{align*} for some $C \in \CC$ (we do not need $C$ because there is no element $f_{k, -n_0 + 1}$). Thus, substituting $\tau$ for $z$, and writing this sum in terms of our canonical basis elements, we see that \begin{align*}
		f_{k, -n_0 + 2} (\tau) =& f(\tau) f_{k, -n_0}(\tau) - \big( a_{k}(-n_0 , n_0 + 2) + a_{-1} a_{k}(-n_0 , n_0 + 1) \big) f_{k, -n_0} (\tau) \\
		=& f(\tau) f_{k, -n_0}(\tau) + \big( b_{2 - k}(n_0 + 2, -n_0) + a_{-1} b_{2 - k}(n_0 + 1, -n_0) \big) f_{k, -n_0} (\tau).
	\end{align*}

	Similarly, for $n \geq - n_0$ ($n \neq -n_0 + 1$) we have \begin{align*}
		f_{k,n + 2}(\tau) =& f(\tau) f_{k, n}(\tau) - a_{n + n_0} f_{k, -n_0} (\tau) - \sum\limits_{i = - 1}^{n + n_0 - 2} a_i f_{k, n - i} (\tau) \\
		&+ b_{2 - k}(n_0 - 1, n) f_{k, -n_0 + 3} (\tau) + a_{-1} b_{2 - k}(n_0 - 1, n) f_{k, -n_0 + 2}(\tau) \\
		&+ \big( b_{2 -k}(n_0 + 2, n) + a_{-1} b_{2 - k}(n_0 + 1, n) + a_1 b_{2 - k}(n_0 - 1, n) \big) f_{k, -n_0} (\tau).
	\end{align*}

	Putting this recurrence relation into the generating function, we have \begin{align*}
		F_k(z,\tau) =& F_k = q^{-n_0} f_{k, -n_0}(\tau) + q^{-n_0 + 2} f_{k, -n_0 + 2}(\tau) + q^{-n_0 + 3} f_{k, -n_0 + 3}(\tau) + q^2 \sum\limits_{n= -n_0 + 4}^{\infty} f_{k,n}(\tau) q^{n - 2} \\
		=& q^{-n_0} f_{k, -n_0}(\tau) + q^{-n_0 + 3} f_{k, -n_0 + 3}(\tau) + q^{-n_0 + 2} f(\tau) f_{k, -n_0} (\tau) \\
		&+ q^{-n_0 + 2} \big( f_{k, -n_0 + 2}(\tau) - f(\tau) f_{k, -n_0} (\tau) \big) +  q^2 \sum\limits_{n= -n_0 + 2}^{\infty} f_{k,n + 2}(\tau) q^n \\
		=& q^{-n_0} f_{k, -n_0}(\tau) + q^{-n_0 + 3} f_{k, -n_0 + 3}(\tau) + q^{-n_0 + 2} f(\tau) f_{k, -n_0} (\tau) \\
		&+ q^{-n_0 + 2} f_{k, -n_0} (\tau) \big( b_{2 - k}(n_0 + 2, -n_0) + a_{-1} b_{2 - k}(n_0 + 1, -n_0) \big) \\
		&+ q^2 \sum\limits_{n= -n_0 + 2}^{\infty} \Big( f(\tau) f_{k, n}(\tau) - a_{n + n_0} f_{k, -n_0} (\tau) - \sum\limits_{i = - 1}^{n + n_0 - 2} a_i f_{k, n - i} (\tau) \\
		&+ b_{2 - k}(n_0 - 1, n) f_{k, -n_0 + 3} (\tau) + a_{-1} b_{2 - k}(n_0 - 1, n) f_{k, -n_0 + 2}(\tau) \\
		&+ \big( b_{2 -k}(n_0 + 2, n) + a_{-1} b_{2 - k}(n_0 + 1, n) + a_1 b_{2 - k}(n_0 - 1, n) \big) f_{k, -n_0} (\tau) \Big) q^n.
		\end{align*} Distributing $q^n$ over the infinite sum and simplifying, we obtain \begin{align*}
		F_k =& q^{-n_0} f_{k, -n_0}(\tau) + q^{-n_0 + 3} f_{k, -n_0 + 3}(\tau) + q^2 f(\tau) \big(q^{-n_0 + 2} f(\tau) f_{k, -n_0} (\tau) + \sum\limits_{n= -n_0 + 2}^{\infty} f_{k, n}(\tau) q^n \big) \\
		&- q^2 f_{k, -n_0} (\tau) \sum\limits_{n= -n_0 + 2}^{\infty} a_{n + n_0} q^n - q^2 \sum\limits_{n= -n_0 + 2}^{\infty} \sum\limits_{i = - 1}^{n + n_0 - 2} a_i f_{k, n - i} (\tau) q^n \\
		&+ q^2 f_{k, -n_0 + 3} (\tau) \sum\limits_{n= -n_0 + 2}^{\infty} b_{2 - k}(n_0 - 1, n) q^n + q^2 a_{-1} f_{k, -n_0 + 2}(\tau) \sum\limits_{n= -n_0 + 2}^{\infty} b_{2 - k}(n_0 - 1, n) q^n \\
		&+ q^2 f_{k, -n_0} (\tau) \big( q^{-n_0} b_{2 - k}(n_0 + 2, -n_0) + \sum\limits_{n= -n_0 + 2}^{\infty} b_{2 -k}(n_0 + 2, n) q^n \big) \\
		&+ q^2 a_{-1} f_{k, -n_0} (\tau) \big( q^{-n_0} b_{2 - k}(n_0 + 1, -n_0) + \sum\limits_{n= -n_0 + 2}^{\infty} b_{2 - k}(n_0 + 1, n) q^n \big) \\
		&+ q^2 a_1 f_{k, -n_0} (\tau) \sum\limits_{n= -n_0 + 2}^{\infty} b_{2 - k}(n_0 - 1, n) q^n.
	\end{align*}
	
	Consider the double sum. We have \begin{align*}
		\sum\limits_{n= -n_0 + 2}^{\infty} \sum\limits_{i = - 1}^{n + n_0 - 2} a_i f_{k, n - i} (\tau) q^n =& \sum\limits_{n= 0}^{\infty} \sum\limits_{i = - 1}^n a_i f_{k, -n_0 + 2 + n - i} (\tau) q^{-n_0 + 2 + n} \\
		=& q^{-1} a_{-1} \big( F_k - q^{-n_0 + 2} f_{k, -n_0 + 2} (\tau) - q^{-n_0} f_{k, -n_0} (\tau) \big) \\
		&+ \big( f(z) - q^{-1} a_{-1} - q^{-2} \big) \big( F_k - q^{-n_0} f_{k, -n_0} (\tau) \big) \\
		=& f(z) F_k - q^{-2} F_k - q^{-n_0 + 1} a_{-1} f_{k, -n_0 + 2} (\tau) \\
		&- q^{-n_0} f(z) f_{k, -n_0} (\tau) + q^{-n_0 - 2} f_{k, -n_0} (\tau).
	\end{align*}
	
	Plugging this back into the original expression, and rewriting the other sums by using the definitions of $F_k$ and of $g_{k, n} (z)$, we have \begin{align*}
		F_k =& q^{-n_0} f_{k, -n_0}(\tau) + q^{-n_0 + 3} f_{k, -n_0 + 3}(\tau) + q^2 f(\tau) F_k - q^{-n_0 + 2} f_{k, -n_0} (\tau) \big( f(z) - q a_1 - q^{-1} a_{-1} - q^{-2} \big) \\
		&- q^2 \big( f(z) F_k - q^{-2} F_k - q^{-n_0 + 1} a_{-1} f_{k, -n_0 + 2} (\tau) - q^{-n_0} f(z) f_{k, -n_0} (\tau) + q^{-n_0 - 2} f_{k, -n_0} (\tau) \big) \\
		&+ q^2 f_{k, -n_0 + 3} (\tau) \big( g_{2 - k, n_0 - 1} (z) - q^{-n_0 + 1} \big) + q^2 a_{-1} f_{k, -n_0 + 2}(\tau) \big( g_{2 - k, n_0 - 1} (z) - q^{-n_0 + 1} \big) \\
		&+ q^2 f_{k, -n_0} (\tau) \big( g_{2 - k, n_0 + 2} (z) - q^{-n_0 - 2} \big) + q^2 a_{-1} f_{k, -n_0} (\tau) \big( g_{2 - k, n_0 + 1} (z) - q^{-n_0 - 1} \big) \\
		&+ q^2 a_1 f_{k, -n_0} (\tau) \big( g_{2 - k, n_0 - 1} (z) - q^{-n_0 + 1} \big) \\
		=& F_k + q^2 \big( f(\tau) - f(z) \big) F_k + q^2 f_{k, -n_0 + 3} (\tau) g_{2 - k, n_0 - 1} (z) \\
		&+ q^2 a_{-1} f_{k, -n_0 + 2}(\tau) g_{2 - k, n_0 - 1} (z) + q^2 f_{k, -n_0} (\tau) g_{2 - k, n_0 + 2} (z) \\
		&+ q^2 a_{-1} f_{k, -n_0} (\tau) g_{2 - k, n_0 + 1} (z) + q^2 a_1 f_{k, -n_0} (\tau) g_{2 - k, n_0 - 1} (z).
	\end{align*} Now subtracting $F_k$ and $q^2 (f(\tau) - f(z)) F_k$ from both sides, and dividing by $q^2$, we obtain \begin{align*}
		\big( f(z) - f(\tau) \big) F_k =& \big( a_1 g_{2 - k, n_0 - 1} (z) + a_{-1} g_{2 - k, n_0 + 1} (z) + g_{2 - k, n_0 + 2} (z) \big) f_{k, -n_0} (\tau) \\
		&+ a_{-1} g_{2 - k, n_0 - 1} (z) f_{k, -n_0 + 2}(\tau) + g_{2 - k, n_0 - 1} (z) f_{k, -n_0 + 3} (\tau) \\
	\end{align*} and thus \begin{align*}
		F_k(z, \tau) =& \frac{\big( a_1 g_{2 - k, n_0 - 1} (z) + a_{-1} g_{2 - k, n_0 + 1} (z) + g_{2 - k, n_0 + 2} (z) \big) f_{k, -n_0} (\tau)}{f(z) - f(\tau)} \\
		&+ \frac{a_{-1} g_{2 - k, n_0 - 1} (z) f_{k, -n_0 + 2}(\tau) + g_{2 - k, n_0 - 1} (z) f_{k, -n_0 + 3} (\tau)}{f(z) - f(\tau)} \\
		=& \frac{\big( b_1 f_{k, -n_0} (\tau) + b_{-1} f_{k, -n_0 + 2}(\tau) + f_{k, -n_0 + 3} (\tau) \big) g_{2 - k, n_0 - 1} (z)}{g(z) - g(\tau)} \\
		&+ \frac{b_{-1} f_{k, -n_0} (\tau) g_{2 - k, n_0 + 1} (z) + f_{k, -n_0} (\tau) g_{2 - k, n_0 + 2} (z)}{g(z) - g(\tau)}
	\end{align*} as claimed. The second equality holds because $f(z) - g(z)$ is a constant.
\end{proof}

\bibliographystyle{amsplain}
%\bibliography{JenkinsMolnarDualitySources}{}

\end{document}